\documentclass{siamltex}
\usepackage{lineno,hyperref}
\usepackage{multirow}
\usepackage[english]{babel}
\usepackage{amsmath}
\usepackage{bm}
\usepackage[matha,mathx]{mathabx}
\usepackage{mathrsfs}
\usepackage{amssymb}
\usepackage{algpseudocode} 
\usepackage{algorithm}
\usepackage{graphicx}
\usepackage{color} 
\usepackage{caption}
\usepackage{enumitem} 

\makeatletter
\newcommand*{\rom}[1]{\expandafter\@slowromancap\romannumeral #1@}
\makeatother

\newcommand{\R}{\mathbb{R}}

\DeclareMathOperator{\Tr}{Tr}

\usepackage[table]{xcolor}
\definecolor{lightgray}{gray}{0.9}
\usepackage{arydshln}
\usepackage{diagbox}
\newtheorem{remark}{Remark}
\usepackage{subcaption}

\usepackage{pgfplots}
\pgfplotsset{compat=1.17}

\usepackage{here}

\DeclareMathAlphabet{\mathpzc}{OT1}{pzc}{m}{it}
\title{Trace Ratio vs Ratio Trace Methods for Multidimensional Dimensionality Reduction}

\author{ A. Zahir\thanks{The UM6P Vanguard Center, Mohammed VI Polytechnic University, Green 
City, Morocco.} \and F. Dufrenois \thanks{Université du Littoral Côte d'Opale, LISIC, 50 rue F. Buisson, 62228 Calais-Cedex, France.}
\and K. Jbilou\thanks{Université du Littoral Côte d'Opale, LMPA, 50 rue F. Buisson, 62228 Calais-Cedex, France.}
\and A. Ratnani\footnotemark[1]}

\begin{document}
\maketitle

\begin{abstract}
We propose a higher-order dimensionality reduction framework based on the Trace Ratio (TR) optimization problem. We establish conditions for existence and uniqueness of solutions and clarify the theoretical connection between the Trace Ratio and its surrogate, the Ratio Trace (RT) formulation. Building on these foundations, we design a Newton-type iterative algorithm that operates directly in the tensor domain via the Einstein product, avoiding data flattening and preserving multi-dimensional structure.

This approach extends classical Linear Discriminant Analysis (LDA) to higher-order tensors, offering a natural generalization of trace-based dimensionality reduction from matrices to tensors. Numerical experiments on several benchmark datasets confirm the efficiency and robustness of the proposed methods, showing consistent improvements over existing matrix- and tensor-based techniques.
\end{abstract}


\section{Introduction}\label{sec:intro}
Many applications in machine learning, signal and image processing, and pattern recognition \cite{fukunaga2013introduction} involve the analysis of data in high-dimensional spaces. However, operating in such spaces poses significant challenges, including increased computational complexity, greater storage demands. To mitigate these issues and facilitate more efficient model training, researchers frequently turn to dimensionality reduction (DR) techniques. Dimensionality reduction alleviates computational and storage burdens. It also improves interpretability and generalization. Empirical evidence suggests that many high-dimensional datasets have an underlying low-dimensional structure, often well-approximated in manifolds or other compact representations. This observation motivates the search for effective methods that can uncover such intrinsic structures, enabling the construction of lower-dimensional embeddings that retain the key characteristics of the original data. By doing so, we enhance learning efficiency and maintain predictive accuracy while gaining deeper insights into the fundamental patterns in the data.
\\
Dimensionality reduction encompasses a variety of techniques, each tailored to specific data characteristics and analytical goals. These methods can be categorized into linear and nonlinear approaches. 
Linear dimensionality reduction methods operate under the assumption that the data lies within a linear subspace of the high-dimensional space. A classic example is Principal Component Analysis (PCA), which seeks to identify directions (principal components) that capture the maximum variance in the data. These methods are typically computationally efficient, mathematically well-founded, and easy to interpret, making them widely applicable across domains. However, linear techniques may fall short when applied to datasets characterized by nonlinear or manifold structures, where important relationships are not aligned with linear subspaces. 
To address this limitation, nonlinear dimensionality reduction methods have been developed to capture nonlinear patterns in the data. Techniques such as t-Distributed Stochastic Neighbor Embedding (t-SNE) and Isomap aim to preserve local neighborhood relationships and uncover the underlying structure of the data manifold. These methods are particularly well-suited for data visualization and exploratory analysis in high dimensions. Nonetheless, they tend to be computationally intensive and often require careful parameter tuning to yield meaningful results, potentially limiting their scalability in large datasets.
\\
Dimensionality reduction methods can also be classified based on the availability of label information, distinguishing between supervised and unsupervised approaches. Supervised methods exploit labeled data to enhance class separability, improving performance in tasks such as classification. In contrast, unsupervised methods aim to uncover inherent structures in the data without relying on labels. 
In recent years, several efforts have been made to extend classical dimensionality reduction techniques, such as Linear Discriminant Analysis (LDA) and Principal Component Analysis (PCA), to settings involving multi-dimensional (tensor) data. Many of these approaches are based on the tensor t-product, which is inherently limited to third-order tensors, restricting their applicability to higher-order data structures \cite{el2022tensor,kilmer2008third,wang2007trace}. Although, a generalization of this product to higher-order tensors exists. It is not widely adopted in the literature.\\ 
To address this limitation, recent work by \cite{zahir2024higher} introduces a unified framework that generalizes both linear and nonlinear dimensionality reduction methods—including PCA, ONPP, NPP, and OLPP—using the Einstein product, which naturally supports higher-order tensors. This framework encompasses supervised extensions, kernel-based formulations, and out-of-sample generalizations for nonlinear methods, offering a flexible and scalable solution for high-dimensional data analysis.
\\
An earlier contribution \cite{dufrenois2023multilinear} focuses specifically on the Linear Discriminant Analysis (LDA), introducing two tensor-based extensions under the name Multilinear Discriminant Analysis (MLDA). These variants leverage distinct tensor operations, namely the L-product family and the n-mode product.
Some recent optimization techniques are based on the \textit{Trace Ratio} (TR) problem and numerous methods have been proposed to solve this problem which is known to be nonconvex and challenging to solve. To address this challenge, a related—but not equivalent—formulation known as the \textit{Ratio Trace} (RT) problem has been commonly used in the literature as a simplification. 
However, these approaches typically require flattening multi-dimensional data into matrices, which can be problematic. This transformation may lead to loss of the intrinsic structure of the original data—an aspect that is often crucial for achieving accurate and meaningful results.
\\
Tensor-based methods have recently gained significant attention, as they offer a natural generalization of many matrix-based frameworks. Among the most widely used tensor products are the t-product, introduced in \cite{kilmer2008third}, and the Einstein product \cite{Brazell2013}. These approaches have been successfully applied across a range of domains, including denoising, tensor completion, classification, clustering \cite{zahir2025multilinear}, and \textit{dimensionality reduction} (DR).

The main contributions of this section are:
\begin{itemize}
\item We formulate higher-order Trace Ratio (TR) and Ratio Trace (RT) criteria using the Einstein product, preserving tensor structure beyond third order.
\item We establish existence and uniqueness results for the TR solution and derive necessary optimality conditions.
\item We design an efficient Newton-type iterative algorithm, together with a regularized variant, and clarify the connection between TR and least-squares regression.
\item We propose a tensor extension of Linear Discriminant Analysis (MDA$_e$), and evaluate its performance on benchmark datasets.
\end{itemize}

The paper is organized as follows. Section~\ref{sec:preliminaries} recalls multilinear algebra and the Einstein product. Section~\ref{sec:trace_ratio} develops the TR and RT problems, their relation, and the proposed algorithm. Section~\ref{sec:proposed_methods} introduces the multilinear discriminant analysis framework. Section~\ref{sec:experiments} reports numerical results, and Section~\ref{sec:conclusion} concludes with future directions.

\section{Preliminaries}\label{sec:preliminaries}
Let $\mathcal{A}\in \R^{I_1\times \cdots \times I_M}$ be an $M$th-order tensor, with entry $\mathcal{A}_{i_1 \ldots i_M}$ at position $(i_1,\ldots,i_M)$. The $i$-th frontal slice of $\mathcal{A}$, denoted by $\mathcal{A}^{(i)}$, is the tensor obtained by fixing the last index to $i$. 

A tensor $\mathcal{A} \in \R^{I_1\times \cdots \times I_M \times J_1\times \cdots \times J_N}$ is called \emph{even} if it verifies $M=N$, and \emph{square} if $I_k=J_k$ for all $k=1,\ldots,M$ \cite{qi2017tensor}. Its transpose is defined by
\[
(\mathcal{A}^T )_{j_1\ldots j_M i_1\ldots i_M}=\mathcal{A}_{i_1\ldots i_M j_1\ldots j_M},
\]
and $\mathcal{A}$ is symmetric if $\mathcal{A}^T=\mathcal{A}$. A diagonal tensor has all zero entries except those with repeated indices $(i_1,\ldots,i_M,i_1,\ldots,i_M)$. If these entries are one, it is the identity tensor $\mathcal{I}_M$. The symmetric part is defined by $\operatorname{sym}(\mathcal{A})=\tfrac{1}{2}(\mathcal{A}+\mathcal{A}^T)$. 
\medskip
\begin{definition}[m-mode product]
Let $\mathcal{A} \in \R^{I_1 \times \cdots \times I_M}$ and $Z \in \R^{J\times I_m}$. 
The $m$-mode product is
\[
(\mathcal{A} \times_m Z)_{i_1 \ldots i_{m-1} j i_{m+1} \ldots i_M}
=\sum_{i_m=1}^{I_m} Z_{j i_m} \,\mathcal{A}_{i_1 \ldots i_m \ldots i_M}.
\]
\end{definition}

\begin{definition}[Einstein product {\cite{Brazell2013}}]
Let $\mathcal{A} \in \R^{I_1 \times \cdots \times I_M \times K_1 \times \cdots \times K_N}$ and\\ $\mathcal{Z} \in \R^{K_1 \times \cdots \times K_N \times J_1 \times \cdots \times J_M}$. The Einstein product of $\mathcal{A}$ and $\mathcal{Z}$ is defined as
\[
\left(\mathcal{A} *_N \mathcal{Z}\right)_{i_1 \ldots i_M j_1 \ldots j_M}
=\sum_{k_1 \ldots k_N} \mathcal{A}_{i_1 \ldots i_M k_1 \ldots k_N} \,\mathcal{Z}_{k_1 \ldots k_N j_1 \ldots j_M}.
\]
\end{definition}
\noindent
When $\mathcal{A} \in \R^{I_1 \times \cdots \times I_M \times J}$ and $Z \in \R^{J \times K}$, the Einstein product over one mode coincides with the $m$-mode product:
$\mathcal{A} *_1 Z=\mathcal{A} \times_{M+1} Z^T$.

\begin{definition}[Inner product and norm]
Let $\mathcal{A}, \mathcal{Z} \in \R^{I_1 \times \cdots \times I_M}$. Then
\[
\langle \mathcal{A}, \mathcal{Z}\rangle=\Tr (\mathcal{A} *_M \mathcal{Z}^T), \qquad
\|\mathcal{A}\|_{F}=\sqrt{\langle \mathcal{A}, \mathcal{A}\rangle}.
\]
\end{definition}
\medskip
\begin{definition}[Invertibility and definiteness]
A square $2M$th-order tensor $\mathcal{A}$ is invertible if there exists $\mathcal{A}^{-1}$ such that
$\mathcal{A} *_M \mathcal{A}^{-1}=\mathcal{A}^{-1} *_M \mathcal{A}=\mathcal{I}_M$. 
It is unitary if $\mathcal{A}^T *_M \mathcal{A}=\mathcal{I}_M$, 
positive semi-definite if $\langle \mathcal{X},\mathcal{A} *_M \mathcal{X}\rangle \geq 0$ for all $\mathcal{X}$, 
and positive definite if the inequality is strict for $\mathcal{X}\neq 0$.
\end{definition}
\medskip
\begin{definition}[Eigenvalue problems \cite{wang2022generalized}]
Let $\mathcal{A},\mathcal{B} \in \R^{I_1 \times \cdots \times I_M \times I_1 \times \cdots \times I_M}$ two symmetric tensors. We define
\begin{itemize}
\item \emph{Eigenvalue problem (EVP)}: $\mathcal{A} *_M \mathcal{X}=\lambda \mathcal{X}$.
\item \emph{Generalized eigenvalue problem (GEVP)}: $\mathcal{A} *_M \mathcal{X}=\lambda \mathcal{B} *_M \mathcal{X}$.
\end{itemize}
\end{definition}
\medskip
\begin{definition}[Rank and range space]
Let $\mathcal{A} \in \R^{I_1 \times \ldots \times I_M \times J_1 \times \ldots \times J_M}$,
the range space of $\mathcal{A}$ is defined as
$\mathcal{R}(\mathcal{A}):=\{\mathcal{X} *_M \mathcal{Z}, \; \mathcal{Z} \in \R^{J_1 \times \ldots \times J_M}\}$, and the rank of $\mathcal{A}$ is $\operatorname{rank}(\mathcal{A})=\dim(\mathcal{R}(\mathcal{A}))$.

\end{definition}

\begin{lemma}\label{prop:symmetric_XMXt}
\begin{enumerate}[label=\arabic*.]
\item A symmetric $2M$th-order tensor $\mathcal{X}$ is positive semi-definite iff there exists $\mathcal{B}$ such that $\mathcal{X}=\mathcal{B} *_M \mathcal{B}^T$; it is positive definite iff $\mathcal{B}$ is invertible.
\item If $M$ is a positive semi-definite matrix, then $\mathcal{X} \times_{M+1} M *_1 \mathcal{X}^T$ is also positive semi-definite.
\item The eigenvalues of a symmetric tensor are real; they are nonnegative if the tensor is semi-definite, and strictly positive if it is positive definite.
\end{enumerate}
\end{lemma}

\begin{theorem}[Einstein Tensor Spectral Theorem {\cite{zahir2024higher}}]
If $\mathcal{A}$ is a symmetric $2M$th-order tensor. Then there exist a unitary tensor $\mathcal{Q}$ and a diagonal tensor $\mathcal{D}$ such that $\mathcal{A}=\mathcal{Q} *_M \mathcal{D} *_M \mathcal{Q}^T$.
\end{theorem}

As in the matrix setting, the generalized eigenvalue problem
\[
\mathcal{A} *_M \mathcal{X}=\lambda \,\mathcal{M} *_M \mathcal{X}, \qquad \mathcal{M}\ \text{invertible},
\]
is equivalent to the standard eigenvalue problem $\widehat{\mathcal{A}} *_M \widehat{\mathcal{X}}=\lambda \widehat{\mathcal{X}}$, with $\widehat{\mathcal{A}}=\mathcal{A} *_M \mathcal{M}^{-1}$ and $\widehat{\mathcal{X}}=\mathcal{M} *_M \mathcal{X}$.

\section{The Trace Ratio and Ratio Trace Problems}\label{sec:trace_ratio}

We now introduce the Trace Ratio (TR) problem in the tensor setting, establish its properties, and present an iterative algorithm. We then discuss its surrogate, the Ratio Trace (RT) problem, which is easier to solve but not equivalent.

\subsection{Trace Ratio problem}

The tensor TR problem is defined as
\begin{equation}
\label{eq:trace_ratio}
\max_{\substack{\mathcal{P} \in \R^{I_1 \times \cdots \times I_M \times d}\\ \mathcal{P}^T *_M \mathcal{C} *_M \mathcal{P}=\mathcal{I}}} 
\left \{ \mathcal{J}_{tr}(\mathcal{P}):=\frac{\Tr\left(\mathcal{P}^T *_M \mathcal{A} *_M \mathcal{P}\right)}{\Tr\left(\mathcal{P}^T *_M \mathcal{B} *_M \mathcal{P}\right)} \right \},
\end{equation}
where $\mathcal{A}$ is symmetric, and $\mathcal{B},\mathcal{C}$ are symmetric positive definite tensors. The variable $\mathcal{P}$ is the projection tensor. 

Without loss of generality, we can assume $\mathcal{C}=\mathcal{I}$, since one may replace
\[
\hat{\mathcal{A}}=\hat{\mathcal{C}} *_M \mathcal{A} *_M \hat{\mathcal{C}}^T, 
\quad 
\hat{\mathcal{B}}=\hat{\mathcal{C}} *_M \mathcal{B} *_M \hat{\mathcal{C}}^T,
\quad 
\mathcal{C}=\hat{\mathcal{C}}^T *_M \hat{\mathcal{C}},
\]
preserving symmetry. Thus the problem is a natural generalization of the classical matrix TR problem.\\
First, we start by proving the existence of the solution, then implement the algorithm before giving the necessary condition for optimality.

\subsubsection{Existence and uniqueness of the solution}
In this part, we use properties seen in the first part to show the existence and uniqueness of the solution of the TR problem.
\medskip

The following lemma shows the case when the denominator of the TR problem is nonnegative where $\mathcal{B}$ is only positive semi-definite.
\begin{lemma}\label{lem:trace_ratio_positive}
Let $\mathcal{B}$ positive semi-definite and assume $\dim \mathcal{N}(\mathcal{B})< d$. Then we have \\$\operatorname{Tr}\!\left( \mathcal{P}^T *_M \mathcal{B} *_M \mathcal{P}\right)>0$ for all unitary $\mathcal{P}\in\R^{I_1 \times \ldots \times I_M \times d}$.
\end{lemma}
\begin{proof}
Let $\mathcal{B}=\mathcal{V}^T *_M \mathcal{D} *_M \mathcal{V}$ be the eigenvalue decomposition, then $$\Tr\left(\mathcal{P}^T *_M \mathcal{B} *_M \mathcal{P}\right)=\Tr\left(\mathcal{Q}^T *_M \mathcal{D} *_M \mathcal{Q}\right)=\sum_{j=1}^d \sum_{i_1,\ldots,i_n} \mathcal{D}_{i_1,\ldots,i_n,i_1,\ldots,i_n} |\mathcal{Q}_{i_1,\ldots,i_n}^{(j)}|^2,$$
with $\mathcal{Q}=\mathcal{V} *_M \mathcal{P}=\left[\mathcal{Q}^{(1)},\mathcal{Q}^{(2)}, \ldots, \mathcal{Q}^{(d)}\right]$ has at least one sub-tensor that is non-singular. Thus, from the equation above, at least one of the nonzero diagonal entries of $\mathcal{D}$ will coincide with one of the norms $|\mathcal{Q}_{i_1,\ldots,i_n}^{(j)}|$, hence, the sum cannot be zero.
\end{proof}
\noindent
The above lemma gives us an appropriate setup for the problem to be well-posed, in this case, the maximum is finite.
\medskip
\begin{proposition}
Under the conditions of Lemma \ref{lem:trace_ratio_positive}, the TR problem has a finite maximum (resp., minimum) value, denoted as $\rho^*$.
\end{proposition}
\begin{proof}
The proof is straightforward, using \ref{lem:trace_ratio_positive}, the denominator of the TR is nonnegative, the Stiefel manifold is compact, and with the continuity of the trace, the maximum value is finite and attained.
\end{proof}

\noindent
Consider the pencil as a function of $\rho$ as follows
\begin{equation*}
g(\rho):=\mathcal{A} - \rho \mathcal{B},
\end{equation*}
Given $\rho^*$ that maximizes the TR problem, we have
\begin{equation*}
\dfrac{\Tr\left((\mathcal{P}^T *_M \mathcal{A} *_M \mathcal{P}\right)}{\Tr\left(\mathcal{P}^T *_M \mathcal{B} *_M \mathcal{P}\right)} \leq \rho^*, \forall \mathcal{P}^T *_M \mathcal{P}=\mathcal{I}.
\end{equation*}
\noindent
This allows us to find the necessary condition for the pair $(\mathcal{P}^*,\rho^*)$ that maximizes \ref{eq:trace_ratio}, by the following
\begin{equation} \label{eq:trace_ratio_necessary}
\max_{\mathcal{P}^T *_M \mathcal{P}} 
\Tr\left(\mathcal{P}^T *_M g(\rho^*) *_M \mathcal{P}\right)= \Tr\left(\mathcal{P}^*{^T} *_M g(\rho^*) *_M \mathcal{P}^*\right)=0.
\end{equation}
If $\rho^*$ is the maximum value of the TR problem, then the $d$ largest eigenvalue pencil $g(\rho^*)$ are zeros, and the corresponding set of eigen-tensors characterize $\mathcal{P}^*$. As a consequence, the solution of the TR problem is unique up to unitary transformation.
\subsubsection{Implementation of the iterative algorithm}
Consider the function
\begin{equation*}
f(\rho):= \max_{\mathcal{P}^T *_M \mathcal{P}= \mathcal{I}} \Tr\left(\mathcal{P}^T *_M g(\rho) *_M \mathcal{P}\right).
\end{equation*}
This function is equal to the sum of the $d$ largest eigenvalues of $g(\rho)$.
\begin{lemma}\label{lem:f_properties}
Assume that $\mathcal{B}$ satisfies the condition of Lemma \ref{lem:trace_ratio_positive}. Then the function $f(\rho)$ posseses the following properties:
\begin{enumerate}[label=\arabic*.]
\item $f'(\rho)=-\Tr(\mathcal{P}^T *_M \mathcal{B} *_M \mathcal{P})$.
\item $f(\rho)$ is strictly decreasing and convex.
\item $f(\rho)=0$ if and only if $\rho=\rho^*$.
\end{enumerate}
\end{lemma}
\begin{proof}

$1.$ Provided that $\mathcal{P}(\rho)$ is differentiable, and diagonalizes $g(\rho)$, i.e., $g(\rho) *_M \mathcal{P}(\rho)=\mathcal{P}(\rho) \times_{M+1} \Delta(\rho)$, then we have
\begin{equation*}
\begin{aligned}
\frac{d}{d \rho} \left[ \mathcal{P}^T *_M g(\rho) *_M \mathcal{P} \right] &= \frac{d }{d \rho} \left[ \mathcal{P}^T *_M \mathcal{A} *_M \mathcal{P}\right] - \frac{d }{d \rho} \left[\rho \mathcal{P}^T *_M \mathcal{B} *_M \mathcal{P} \right]\\
&=2 \operatorname{sym} \left(\frac{d \mathcal{P}^T}{d \rho} *_M \mathcal{A}*_M \mathcal{P}\right) - \frac{d \mathcal{P}^T}{d \rho} *_M \rho \mathcal{B} *_M \mathcal{P}\\
&\quad \quad - \mathcal{P}^T *_M \left( \rho \mathcal{B} *_M \frac{d \mathcal{P}}{d \rho} + \mathcal{B} *_M \mathcal{P} \right)\\
&= 2 \operatorname{sym} \left(\frac{d \mathcal{P}^T}{d \rho} *_M \mathcal{P} \times_{M+1} \Delta(\rho)\right) - \mathcal{P}^T *_M \mathcal{B} *_M \mathcal{P}.
\end{aligned}
\end{equation*}
Computing the derivative of the constraint, we can obtain that $\operatorname{Diag}(\mathcal{P}^T *_M \frac{d\mathcal{P}}{d \rho})=0$, thus, we can compute the derivative of $f$ as follows
\begin{equation*}
\begin{aligned}
f'(\rho) &= \Tr\left[2 \operatorname{sym} \left(\frac{d \mathcal{P}^T}{d \rho} *_M \mathcal{P} \times_{M+1} \Delta(\rho)\right) - \mathcal{P}^T *_M \mathcal{B} *_M \mathcal{P} \right]\\
&= - \Tr(\mathcal{P}^T *_M \mathcal{B} *_M \mathcal{P}),\\
\end{aligned}
\end{equation*}
where we used the fact that the trace of Einstein product of a diagonal tensor with another tensor that has zeros on its diagonals is zero.\\
$2.$ As $\mathcal{B}$ satisfies \eqref{lem:trace_ratio_positive}, we have $f'(\rho) < 0$ for all $\rho$. Thus $f$ is strictly decreasing.\\
Let $t \in [0,1]$, then
\begin{equation*}
g(t\rho_1+(1-t)\rho_2)=\mathcal{A} - (t\rho_1+(1-t)\rho_2) \mathcal{B} = t(\mathcal{A} - \rho_1 \mathcal{B}) + (1-t)(\mathcal{A} - \rho_2 \mathcal{B})= t g(\rho_1) + (1-t) g(\rho_2).
\end{equation*}
Therefore $f(t\rho_1+(1-t)\rho_2) \leq t f(\rho_1) + (1-t) f(\rho_2)$, as the maximum sum of functions is less than the sum of the maximum of functions, Thus $f$ is convex.\\
$3.$ Since $f'$ is strictly decreasing, it is injective, and we have $f'(\rho^*)=0$. This concludes the proof.
\end{proof}

\noindent
We can now express the Newton's method to solve the TR problem that converges globally and fast to the optimal solution thanks to Lemma $\ref{lem:f_properties}$.
\begin{equation}\label{eq:update_rho}
\rho_{k+1}=\rho_k - f(\rho_k)/f'(\rho_k)=\mathcal{J}_{tr}(\mathcal{P}(\rho_k)).
\end{equation}
The following algorithm is the iterative Newton's algorithm to solve the TR problem. The eigenvalue computing can be performed using the Lanczos method based on the Einstein product. As we approach the solution, a more robust method can be used instead.
\begin{algorithm}[h]
\caption{Solve The TR Problem using the iterative algorithm}
\begin{algorithmic}[1]
\Require $\mathcal{A},\mathcal{B}$, $d$ (target dimension).
\State Initialize a unitary $\mathcal{P} \in \R^{I_1 \times I_2 \times \ldots \times I_M \times d}$, and $k=0$.
\State $\rho_k \leftarrow \mathcal{J}_{tr}(\mathcal{P})$.
\While{Not converged}
\State Compute the $d$ largest eigenvalues of $g(\rho_k)$, and the corresponding eigen-tensors to form $\mathcal{P}$.
\State Compute $\rho_{k+1}$ via \eqref{eq:update_rho}.
\State $k=k+1$.
\EndWhile
\State \textbf{return} $\mathcal{P}$ (Projection space).
\end{algorithmic}
\label{alg:TR_iterative}
\end{algorithm}
\subsubsection{Necessary Conditions for Optimality}
The Lagrangian of the RT problem is given by
\begin{equation}
\mathcal{L}(\mathcal{P},\lambda):=\frac{\Tr\left(\mathcal{P}^T *_M \mathcal{A} *_M \mathcal{P}\right)}{\Tr\left(\mathcal{P}^T *_M \mathcal{B} *_M \mathcal{P}\right)}
-\Tr\left( \Lambda \left(\mathcal{P}^T *_M \mathcal{P}-\mathcal{I}\right)\right),
\end{equation}
where $\Lambda$ is the Lagrange multiplier.\\
The derivative of the Lagrangian with respect to $\mathcal{P}$ is given by
\begin{equation*}
\frac{\partial \mathcal{L}}{\partial \mathcal{P}}=\frac{\Tr\left(\mathcal{P}^T *_M \mathcal{A} *_M \mathcal{P}\right) \mathcal{A} *_M \mathcal{P} - \Tr\left(\mathcal{P}^T *_M \mathcal{B} *_M \mathcal{P}\right) \mathcal{B} *_M \mathcal{P} }{\Tr\left(\mathcal{P}^T *_M \mathcal{B} *_M \mathcal{P}\right)^2} -2 \mathcal{P} \times_{M+1} \operatorname{sym}(\Lambda).
\end{equation*}
The optimal solution $\mathcal{P}^*$ and $\Lambda^*$ satisfy the following equation
\begin{equation*}
g(\rho^*) *_M \mathcal{P}^*= \Tr\left(\mathcal{P}^{*^T} *_M \mathcal{B} *_M \mathcal{P}^*\right) \mathcal{P}^* \times_{M+1} \operatorname{sym}(\Lambda^*),
\end{equation*}
with $\rho^*=\mathcal{J}_{tr}(\mathcal{P}^*)$. The eigenvalue decomposition of $\operatorname{sym}(\Lambda^*)$ is $VSV^T$, where we can observe that $\Tr(S)=0$. We can rewrite the above equation as an eigenvalue problem
\begin{equation}
\label{eq:ratio_trace_necessary}
g(\rho^*) *_M \mathcal{Q}^* = \mathcal{Q}^* \times_{M+1} S^*,
\end{equation}
where $\Tr(\mathcal{Q}^{*^T} *_M \mathcal{B} *_M \mathcal{Q}^*) S=S^*$, and $\mathcal{Q}^*=\mathcal{P}^* \times_{M+1} V$ is unitary, since 
$$\mathcal{Q}^{*^T} *_M \mathcal{Q}^*=(\mathcal{P}^* \times_{M+1} V)^T *_M( \mathcal{P}^* \times_{M+1} V)= (V^T *_1 \mathcal{P}^{*^T}) *_M( \mathcal{P}^* \times_{M+1} V)= V^T (\mathcal{P}^{*^T} *_M \mathcal{P}) V.$$
As both $\mathcal{P}^*$ and $V$ are unitary, then $\mathcal{Q}^*$ is unitary.\\
The necessary condition for the maximum of the RT problem for the pair $(\mathcal{Q}^*,\rho^*)$ is the equation \eqref{eq:ratio_trace_necessary}, with zero trace of $S^*$.

\subsection{RT problem}
Next, we go back to the TR problem, which is nonconvex and may have multiple stationary points; see, e.g., the matrix-case analyses in \cite{wang2007trace} and references therein.\\
Similar to the matrix case, the problem can be replaced, by the following non-equivalent problem, called the \textit{Ratio Trace} problem
\begin{equation}
\label{eq:ratio_trace}
\max_{\substack{\mathcal{P} \in \R^{I_1 \times \ldots \times I_M \times d} \\ \mathcal{P}^T *_M \mathcal{P}=\mathcal{I}}} 
\left \{ \mathcal{J}_{rt}(\mathcal{P}):=\Tr\left( \left(\mathcal{P}^T *_M \mathcal{B} *_M \mathcal{P}\right)^{-1}\mathcal{P}^T *_M \mathcal{A} *_M \mathcal{P} \right) \right \}.
\end{equation}
In the next part, we will try to give the mathematical background for the solution of the problem
\subsubsection{Solution of the problem}
The next Theorem \ref{thm:tr_pos_def_2} can be used to solve the TR problem. Although easier to solve, the RT problem is not equivalent to the TR problem and may yield solutions that deviate significantly from the optimal solution.
\begin{theorem}\label{thm:tr_pos_def}
Let $\mathcal{A} \in \R^{I_1 \times \ldots \times I_M \times I_1 \times \ldots \times I_M}$ be a symmetric tensor and $\mathcal{B}$ a positive definite tensor of same size. Then
\[ \min_{\substack{\mathcal{P} \in \R^{I_1 \times \ldots \times I_M \times d}\\ \mathcal{P}^T *_M \mathcal{B} *_M \mathcal{P}=\mathcal{I}}} \Tr\left(\mathcal{P}^T *_M \mathcal{A} *_M \mathcal{P}\right)=\lambda_1+\ldots + \lambda_d, \]
where $\lambda_i, i=1, \ldots, d$ correspond the d-smallest eigenvalues of the pair $\{\mathcal{A}, \mathcal{B}\}$.\\
Moreover $\mathcal{P}_*= [\mathcal{Q}_1, \ldots, \mathcal{Q}_d],$ is the minimizer of the problem, formed by the eigen-tensors of the corresponding $\lambda_i$.
\end{theorem}
\begin{proof}
Let the Lagrangian of the problem be defined as
\[
\mathcal{L}(\mathcal{P}, \Lambda) := \Tr(\mathcal{P}^T *_M \mathcal{A} *_M \mathcal{P}) - \Tr\left( \Lambda^T *_M (\mathcal{P}^T *_M \mathcal{B} *_M \mathcal{P} - \mathcal{I}) \right),
\]
where \(\Lambda \in \R^{d \times d}\) is the matrix of Lagrange multipliers.\\
To compute the gradient of \(\mathcal{L}\) with respect to \(\mathcal{P}\), define
\[
f_1(\mathcal{P}) := \Tr(\mathcal{P}^T *_M \mathcal{A} *_M \mathcal{P}), \quad f_2(\mathcal{P}) := \Tr\left( \Lambda^T *_M (\mathcal{P}^T *_M \mathcal{B} *_M \mathcal{P} - \mathcal{I}) \right).
\]

\textbf{Gradient of \(f_1\):}
Using a perturbation \(\mathcal{P} + \varepsilon \mathcal{H}\):
\[
\begin{aligned}
f_1(\mathcal{P} + \varepsilon \mathcal{H}) &= \Tr((\mathcal{P} + \varepsilon \mathcal{H})^T *_M \mathcal{A} *_M (\mathcal{P} + \varepsilon \mathcal{H})) \\
&= f_1(\mathcal{P}) + \varepsilon \Tr(\mathcal{H}^T *_M \mathcal{A} *_M \mathcal{P} + \mathcal{P}^T *_M \mathcal{A} *_M \mathcal{H}) + \varepsilon^2 \Tr(\mathcal{H}^T *_M \mathcal{A} *_M \mathcal{H}) \\
&= f_1(\mathcal{P}) + \varepsilon \Tr(\mathcal{H}^T *_M (\mathcal{A} + \mathcal{A}^T) *_M \mathcal{P}) + o(\varepsilon).
\end{aligned}
\]
Since \(\mathcal{A}\) is symmetric
\[
\dfrac{\partial f_1}{\partial \mathcal{P}}(\mathcal{H}) = \lim_{\varepsilon \to 0} \dfrac{f_1(\mathcal{P} + \varepsilon \mathcal{H}) - f_1(\mathcal{P})}{\varepsilon} = 2 \Tr(\mathcal{H}^T *_M \mathcal{A} *_M \mathcal{P}),
\]
hence
\[
\dfrac{\partial f_1}{\partial \mathcal{P}} = 2 \mathcal{A} *_M \mathcal{P}.
\]

\textbf{Gradient of \(f_2\):}
Similarly,
\[
\begin{aligned}
f_2(\mathcal{P} + \varepsilon \mathcal{H}) &= \Tr\left(\Lambda^T *_M \left((\mathcal{P} + \varepsilon \mathcal{H})^T *_M \mathcal{B} *_M (\mathcal{P} + \varepsilon \mathcal{H}) - \mathcal{I} \right)\right) \\
&= f_2(\mathcal{P}) + \varepsilon \Tr\left(\Lambda^T *_M (\mathcal{H}^T *_M \mathcal{B} *_M \mathcal{P} + \mathcal{P}^T *_M \mathcal{B} *_M \mathcal{H}) \right) + o(\varepsilon) \\
&= f_2(\mathcal{P}) + \varepsilon \Tr\left( \mathcal{H}^T *_M (\mathcal{B} *_M \mathcal{P} *_M \Lambda^T + \mathcal{B}^T *_M \mathcal{P} *_M \Lambda) \right) + o(\varepsilon).
\end{aligned}
\]
Assuming \(\mathcal{B}\) symmetric
\[
\dfrac{\partial f_2}{\partial \mathcal{P}}(\mathcal{H}) = \Tr(\mathcal{H}^T *_M \mathcal{B} *_M \mathcal{P} *_M (\Lambda + \Lambda^T)),
\]
so
\[
\dfrac{\partial f_2}{\partial \mathcal{P}} = \mathcal{B} *_M \mathcal{P} *_M (\Lambda + \Lambda^T).
\]

\textbf{Stationary:} Combining
\[
\dfrac{\partial \mathcal{L}}{\partial \mathcal{P}} = 2 \mathcal{A} *_M \mathcal{P} - \mathcal{B} *_M \mathcal{P} *_M (\Lambda + \Lambda^T) = 0.
\]
Let \(\widehat{\Lambda} := \Lambda + \Lambda^T\), then
\[
\mathcal{A} *_M \mathcal{P} = \mathcal{B} *_M \mathcal{P} *_M \widehat{\Lambda}.
\]
Since \(\widehat{\Lambda}\) is symmetric, it admits a diagonalization
$
\widehat{\Lambda} = \mathcal{Q}^T *_M \mathcal{D} *_M \mathcal{Q}.
$ 
Then
\[
\begin{aligned}
&\mathcal{A} *_M \mathcal{P} = \mathcal{B} *_M \mathcal{P} *_M \mathcal{Q}^T *_M \mathcal{D} *_M \mathcal{Q}, \\
&\mathcal{A} *_M \mathcal{P} *_M \mathcal{Q}^T = \mathcal{B} *_M \mathcal{P} *_M \mathcal{Q}^T *_M \mathcal{D}.
\end{aligned}
\]
Letting \(\widehat{\mathcal{P}} := \mathcal{P} *_M \mathcal{Q}\), we obtain
$
\mathcal{A} *_M \widehat{\mathcal{P}} = \mathcal{B} *_M \widehat{\mathcal{P}} *_M \mathcal{D}.
$\\
Finally, \(\widehat{\mathcal{P}}\) also satisfies the constraint
\[
\widehat{\mathcal{P}}^T *_M \mathcal{B} *_M \widehat{\mathcal{P}} = \mathcal{I},
\]
which completes the proof.
\end{proof}
\begin{lemma}\label{lem:derivative_trace}
For $i=1,2$, let $\mathcal{X}_i \in \R^{I_1 \times \ldots \times I_M \times I_1 \times \ldots \times I_M}$ be a symmetric tensor and $\mathcal{P} \in \R^{I_1 \times \ldots \times I_M \times d}$. Let $\mathcal{S}_i=\mathcal{P}^T *_M \mathcal{X}_i *_M \mathcal{P}$, with $\mathcal{S}_1$ being invertible. Then
\[\partial [ \Tr\left( \mathcal{S}_1^{-1} \mathcal{S}_2\right)]/\partial \mathcal{P}=-2\mathcal{X}_1 *_M \mathcal{P} \times_{M+1} S_1^{-1} \mathcal{S}_2 \mathcal{S}_1^{-1} + 2 \mathcal{X}_2 *_M \mathcal{P} \times_{M+1} \mathcal{S}_1^{-1}. \]
\end{lemma}
\noindent
The proof is straightforward.
\medskip
\begin{theorem}\label{thm:tr_pos_def_2}
Let $\mathcal{A}$ and $\mathcal{B} \in \R^{I_1 \times \ldots \times I_M \times I_1 \times \ldots \times I_M}$ be two tensors with $\mathcal{B}$ being symmetric positive definite. Then
\[\max_{\substack{\mathcal{P} \in \R^{I_1 \times \ldots \times I_M \times d}\\ \mathcal{P}^T *_M \mathcal{P}=\mathcal{I}}} \Tr\left( \left(\mathcal{P}^T *_M \mathcal{B} *_M \mathcal{P}\right)^{-1}\mathcal{P}^T *_M \mathcal{A} *_M \mathcal{P}\right), \] 
is equivalent to solve the following generalized eigenvalue problem 
\begin{equation}\label{geneig}
\mathcal{A} *_M \mathcal{Z}=\lambda \mathcal{B} *_M \mathcal{Z}.
\end{equation}
The solution is given by the $d$ largest eigen-tensors of the pair $\{\mathcal{A}, \mathcal{B}\}$.
\end{theorem}
\medskip
\begin{proof}
By Lemma \ref{lem:derivative_trace}, and solving the first-order optimal condition leads to
$$ \mathcal{B}^{-1} *_M \mathcal{A} *_M \mathcal{P}=\mathcal{P} \times_{M+1} (\mathcal{P}^T *_M \mathcal{B} *_M \mathcal{P})^{-1} \mathcal{P}^T *_M \mathcal{A} *_M \mathcal{P}.$$
Knowing that $\left(\mathcal{P}^T *_M \mathcal{B} *_M \mathcal{P}\right)^{-1}$ and $\mathcal{P}^T *_M \mathcal{A} *_M \mathcal{P}$ can be simultaneously diagonalized as
$$\left(\mathcal{P}^T *_M \mathcal{B} *_M \mathcal{P}\right)^{-1}=Q D_1 Q^T, \quad \mathcal{P}^T *_M \mathcal{A} *_M \mathcal{P}=Q D_2 Q^T,$$ 
then by setting $\mathcal{T}=\mathcal{P} \times_{M+1} Q^T$, and a diagonal matrix $D_{12}=D_1 D_2$, we can write the above equation as
$$
\mathcal{B}^{-1} *_M \mathcal{A} *_M \mathcal{T} =\mathcal{T} \times_{M+1} D_{12}.$$ 
Setting $\lambda$ as the diagonal entries of $D_{12}$, and $\mathcal{Z}$ as the frontal slices of $\mathcal{T}$, we have exactly the generalized eigenvalue problem.
\end{proof}

Thus, Theorem \ref{thm:tr_pos_def_2} provides an immediate solution strategy for $MDA_e^{rt}$: form $(\mathcal{S}_b,\mathcal{S}_w)$, solve the tensor GEVP $\mathcal{S}_b *_M \mathcal{Z}=\lambda\, \mathcal{S}_w *_M \mathcal{Z}$, and collect the $d$ leading eigen-tensors to build $\mathcal{P}$.

The following algorithm shows the iterative algorithm to solve the $MDA_e^{tr}$ problem.

\begin{algorithm}[h]
\caption{Solve The $MDA_e^{tr}$ problem using the iterative algorithm}
\label{algo:lda_e}
\begin{algorithmic}[1]
\Require $\mathcal{X}$ (data), $Y$ (labels), $c$ ($\#$ classes), $d$ (target dimension).
\State Construct the within scatter tensor $\mathcal{S}_w$ and the between scatter tensor $\mathcal{S}_b$.
\State Apply algorithm \ref{alg:TR_iterative} to solve the TR problem with $\mathcal{A}=\mathcal{S}_b$ and $\mathcal{B}=\mathcal{S}_w$.
\State \textbf{return} $\mathcal{P}$ (Projection space).
\end{algorithmic}
\end{algorithm}

\noindent
Note that the within scatter tensor $\mathcal{S}_w$ \eqref{within_tensor} and the between scatter tensor $\mathcal{S}_b$ \eqref{between_tensor} would be described in the next section \ref{sec:proposed_methods}.

Executing the method directly is expensive due to computing the eingenvalues. A least squares equivalence can be used to accelerate the method.
\subsubsection{Least squares}
The relationship between least squares regression and discriminant analysis has been extensively studied in the multivariate statistics and machine learning literature; see, e.g., \cite{ye2007least,lee2015equivalence}. In the classical (matrix) case, it is well known that Fisher’s Linear Discriminant Analysis (LDA) is equivalent to solving a certain multivariate least-squares regression problem, where the response variable is a suitably coded label matrix. Our tensor formulation generalizes this connection: the proposed $MDA$ can be seen as a multivariate regression in the tensor domain, where the predictor is the centered data tensor $\overline{\mathcal{X}}$ and the response is a label-indicator tensor $B$. A regularized least squares problem can be formulated as follows
\begin{equation} \label{eq:ls}
\min_{\mathcal{P} \in \R^{I_1 \times \ldots \times I_M \times c}} \left\| \mathcal{A}^T *_M \mathcal{P} - B\right\|_F^2 + \epsilon \left\| \mathcal{P}\right\|_F^2.
\end{equation}
The solution of \ref{eq:ls} is $\mathcal{P}=(\mathcal{A}^T *_M \mathcal{A} + \epsilon I)^{-1} *_1 \mathcal{A}^T \times_{M+1} B$.\\
The following lemma shows the relationship between the two problems when $\mathcal{A}=\overline{\mathcal{X}}$ and $B$ is a centered indicator-matrix, one possible example is the following
\begin{equation*}
Y_{i,j}= \begin{cases}\sqrt{\frac{n}{n_j}}-\sqrt{\frac{n_j}{n}} & \text { if } \mathcal{X}^{(i)} \in C_j \\ -\sqrt{\frac{n_j}{n}} & \text { otherwise }\end{cases} \in \R^{n \times c},
\end{equation*}
We note also that other label matrices can be used, if they verify some conditions, we refer to \cite{lee2015equivalence} for more details.
The following shows the algorithm to solve the LDA using the least square problem. 
\begin{lemma}
$\mathcal{J}_{tr}(\mathcal{P})=\mathcal{J}_{tr}(\mathcal{K} *_M \mathcal{P})$ for any non-singular tensor $\mathcal{K}$.\\
\end{lemma}
The proof is straightforward. The lemma shows the the projection space of the two problems, yet, it does not guarantee that the classification performance is the same.
The lemma reduces the search to the space of the solution, as it does not depend on the basis of the space.
\begin{proposition}
Assume that $\operatorname{rank}(\mathcal{S}_b)+ \operatorname{rank}(\mathcal{S}_w)=\operatorname{rank}(\mathcal{S}_t)$, then the solution of the MDA problem is the same as the solution of the least square problem up to a unitary transformation.\\
\end{proposition}
\noindent
The proof is a straightforward generalization from Theorem (5.1) in \cite{ye2007least}. The mild condition holds in most cases (\cite{ye2007least}).

As K-NN preserves the pairwise distances, hence, it is not affected by the unitary transformation, thus the classification performance should be the same.\\
Note that the equivalence is set when dimension $d$ is $\operatorname{rank}(\mathcal{S}_b)$, and for a specific label matrix, as there are multiple choices.\\
In case we want to reduce the dimension to an arbitrary value $d \leq \operatorname{rank}(\mathcal{S}_b)$, similar to \cite{lee2015equivalence} we can use the least square problem to reduce it first to $c$, then, in a second stage, apply the MDA on this small scatter tensors to reduce it to $d$. A QR decomposition can be used to avoid singularities.
\begin{algorithm}[h]
\caption{Solve The MDA using Least squares}
\label{algo:lda_ls}
\begin{algorithmic}[1]
\Require $\mathcal{X}$ (data), $Y$ (labels), $c$ ($\#$ classes), $d$ (target dimension).
\State Get $B$ from the label matrix $Y$. \Comment{By choosing one of the proposed matrices}
\State Solve \eqref{eq:ls} to get $\mathcal{P}_1$ given centered data $\mathcal{A}=\overline{\mathcal{X}}$.
\State Construct the within scatter tensor $\mathcal{S}_w$ and the between scatter tensor $\mathcal{S}_b$.
\State $[\mathcal{Q}, \mathcal{R}] = \operatorname{QR}( \mathcal{P}_1)$. \Comment{QR decomposition}
\State Solve the GEVP problem $ \left(\mathcal{Q}^\top *_M \mathcal{S}_b *_M \mathcal{Q} \right) P_2 = \Lambda \left(\mathcal{Q}^\top *_M \mathcal{S}_{t}^* *_M \mathcal{Q} \right) P_2$.
\State \textbf{return} $\mathcal{P} \leftarrow \mathcal{Q} \times_{M+1} P_2$.
\end{algorithmic}
\end{algorithm}

\section{Multilinear Discriminant Analysis via Einstein Product}\label{sec:proposed_methods}
The TR problem offers a broad framework for methods of this type, where various approaches to constructing the tensors $\mathcal{A}$, $\mathcal{B}$ can result in different techniques, including unsupervised, supervised, and semi-supervised methods. In this section, we focus on generalizing the most common one; LDA. We introduce the Multilinear Discriminant Analysis via Einstein product ($MDA_e$). It is a supervised method that seeks a projection tensor maximizing the ratio between between-class and within-class scatter.

Given a data tensor $\mathcal{X} \in \R^{I_1 \times \ldots \times I_M \times n}$ of $n$ points and a label matrix $Y$. We aim to find a linear projection tensor $\mathcal{P} \in \R^{I_1 \times \ldots \times I_M \times d}$, where $d$ is the desired reduced feature dimension of the solution, the proposed method is Linear, and Supervised.
\subsection{Formulation}
Given a data tensor $\mathcal{X} \in \R^{I_1 \times \cdots \times I_M \times n}$ of $n$ samples and a label matrix $Y$, the goal is to compute a projection tensor $\mathcal{P} \in \R^{I_1 \times \cdots \times I_M \times d}$, where $d$ is the \emph{target dimension} (typically $d \ll I_1 \cdots I_M$), such that the projected features maximize class separability. 

MDA$_e$ is defined by the TR criterion with
\[
\mathcal{A}=\mathcal{S}_b, \qquad \mathcal{B}=\mathcal{S}_w,
\]
where $\mathcal{S}_b$ and $\mathcal{S}_w$ are the between- and within-class scatter tensors (defined below). 
Two variants are obtained:
\begin{itemize}
\item $MDA_e^{tr}$: TR formulation, solved by the iterative algorithm (Algorithm \ref{alg:TR_iterative}).
\item $MDA_e^{rt}$: RT surrogate, solved by generalized eigenvalue decomposition.
\end{itemize}
\subsection{Between, within, and total scatter tensors}
Denote $C_i$ the set of samples of the $i^{th}$ class. The within scatter $\mathcal{S}_w$ tensor is defined as
\begin{equation}\label{within_tensor}
\mathcal{S}_w=\sum_{i=1}^{c} \sum_{j\in C_i} (\mathcal{X}^{(j)}-\xi_i) *_1 (\mathcal{X}^{(j)}-\xi_i)^T \in \R^{I_1 \times \ldots \times I_M \times I_1 \times \ldots \times I_M},
\end{equation}
where $\xi_i \in \R^{I_1 \times \ldots \times I_M \times 1}$ is the mean tensor of class $i$. 
Equivalently, using the centering matrix $H=I_n-\tfrac{1}{n}\mathbf{1}\mathbf{1}^T$, 
\[
\mathcal{S}_w=\overline{\mathcal{X}} *_1 \overline{\mathcal{X}}^T, \qquad \overline{\mathcal{X}}=\mathcal{X} \times_{M+1} H.
\]
The trace of $\mathcal{S}_w$ measures the within-class cohesion.

Given $n_i$ as the number of samples in each class, which gives us ($\sum n_i=n$). The between scatter tensor $\mathcal{S}_b$ is defined as
\begin{equation}\label{between_tensor}
\mathcal{S}_b=\sum_{i=1}^{c} n_i (\xi_i-\xi) *_1 (\xi_i-\xi)^T \in \R^{I_1 \times \ldots \times I_M \times I_1 \times \ldots \times I_M},
\end{equation}
where $\xi$ is the global mean tensor. 
It can also be expressed as
\[
\mathcal{S}_b=\mathcal{X} \times_{M+1} C_b *_1 \mathcal{X}^T,
\]
with $C_b=H W_b^T W_b H$ and $W_b=(WW^T)^{-1/2}W$, where $W$ collects the class-indicator vectors.

The trace of $\mathcal{S}_b$ measures the between-class separation.\\
The total scatter tensor is defined as
\begin{equation}
 \mathcal{S}_t=\sum_{j=1}^{n} (\mathcal{X}^{(j)}-\xi) *_M (\mathcal{X}^{(j)}-\xi)^T \in \R^{I_1 \times \ldots \times I_M \times I_1 \times \ldots \times I_M}.
 \end{equation}
It follows from the definitions above that the total scatter tensor is the sum of the within scatter tensor and the between scatter tensor; $\mathcal{S}_t=\mathcal{S}_w+\mathcal{S}_b$.\\
Note that $\mathcal{S}_w$ can always be replaced by $\mathcal{S}_t$ in the $\mathcal{J}_{tr}$ problem without affecting the result. This substitution is preferable since it bounds the value of the TR problem by 1.\\

\subsection{Subspace reduction}
We can reduce the search to a smaller subspace, using the fact that the null space of $\mathcal{S}_t$ has no contribution on the $\mathcal{J}_{tr}$ problem value.\\
This can be proved as follows. Denote by $\mathcal{Z}_1$ the orthogonal basis of $\mathcal{N}(\mathcal{S}_t)$, and $\mathcal{Z}_1$ the orthogonal basis of $\mathcal{N}^\perp(\mathcal{S}_t)$, then for any $\mathcal{P} \in \R^{I_1 \times \ldots \times I_M \times d}$, there is $W_1,W_2$ such that 
\begin{equation*}
\mathcal{P}= \mathcal{Z}_1 \times_{M+1} W_1 + \mathcal{Z}_2 \times_{M+1} W_2= \mathcal{P}_1 + \mathcal{P}_2.
\end{equation*}
Since $\mathcal{S}_t,\mathcal{S}_b,$ and $\mathcal{S}_w$ are symmetric positive semi-definite, $\mathcal{N}(\mathcal{S}_t) \subset \mathcal{N}(\mathcal{S}_b)$, therefore 
\begin{equation*}
\mathcal{S}_b *_M \mathcal{P}= \mathcal{S}_b *_M \mathcal{P}_2, \quad \mathcal{S}_t *_M \mathcal{P}= \mathcal{S}_t *_M \mathcal{P}_2, 
\end{equation*}
thus, we have 
\begin{equation*}
\mathcal{P}^T *_M \mathcal{S}_b *_M \mathcal{P}= \mathcal{P}_2^T *_M \mathcal{S}_b *_M \mathcal{P}_2, \quad \mathcal{P}^T *_M \mathcal{S}_t *_M \mathcal{P}= \mathcal{P}_2^T *_M \mathcal{S}_t *_M \mathcal{P}_2.
\end{equation*}
Finally, we have the following,
\begin{equation*}
\max_{\substack{\mathcal{P} \in \R^{I_1 \times \ldots \times I_M \times d}\\ \mathcal{P}^T *_M \mathcal{P}=\mathcal{I}}}\frac{\Tr(\mathcal{P}^T *_M \mathcal{S}_b *_M \mathcal{P})}{\Tr(\mathcal{P}^T *_M \mathcal{S}_w *_M \mathcal{P})} = \max_{\substack{\mathcal{P} \in \R^{I_1 \times \ldots \times I_M \times d}\\ \mathcal{P}^T *_M \mathcal{P}=\mathcal{I}\\
\mathcal{P} \in \mathcal{N}^\perp(\mathcal{S}_t)}}
\frac{\Tr(\mathcal{P}^T *_M \mathcal{S}_b *_M \mathcal{P})}{\Tr(\mathcal{P}^T *_M \mathcal{S}_w *_M \mathcal{P})},
\end{equation*}
which is a generalization of the finding by \cite{park2008comparison}, that the null space of the total scatter matrix/ tensor has no useful information on the TR value.
A regularization is often used in the case of small-sample-size (SSS), 
\subsection{Regularization}
In case of $\mathcal{B}$ is not invertible, which is the case in the $MDA_e$ problem, as the number of classes is generally less than the dimension of the data, we can add a regularization term to the TR problem, to ensure the existence of the solution. The regularized tensor $\mathcal{B}_\epsilon=\mathcal{B}+\epsilon \mathcal{I}$, with $\epsilon \geq 0$. 
Note that other forms of regularization on $\mathcal{B}$ can be used, such as $\mathcal{B}_\epsilon =\epsilon \mathcal{B}+ \mathcal{I}$, or $\mathcal{B}_\alpha =\alpha \mathcal{B}+ (1-\alpha) \mathcal{I}^T$, for $0 \geq \alpha \geq 1$. The value of $\epsilon$ or $\alpha$ should be chosen carefully, and various methods can guide this selection. These methods include the discrepancy principle, which uses information about the noise; The L-curve criterion, which plots the residual norm against the side constraint norm, and generalized cross-validation, aimed at minimizing prediction errors.\\

Note that a regularization technique is also often used in small-sample-size (SSS) scenarios, particularly when dealing with high-dimensional data, which common in images, video... One approach to addressing this issue is the exponential discriminant analysis (EDA) method, introduced in \cite{zhang2009generalized}, which eliminates the need for regularization. The key idea behind EDA is to replace the between-scatter and within-scatter matrices with their exponential counterparts. This transformation ensures that the resulting TR remains non-singular. A similar approach can be adapted, which would result in a new methods based on it.

\begin{remark}
Another method that is similar to LDA, is the Fischer Discriminant Analysis (FDA), that is used in the case of two classes, and the scatter matrices are replaced by the covariance matrices.
\end{remark}

\section{Numerical experiments}\label{sec:experiments}
To show the effectiveness and robustness of the proposed methods, we conduct experiments on several benchmark datasets widely used in the literature. Specifically, we evaluate performance on the GTDB dataset for facial recognition, the MNIST dataset for handwritten digit recognition, and the DIV dataset for multi-modal digit recognition, which includes both visual and auditory data. Unlike datasets that provide pre-processed features, these benchmarks offer raw image data—and in the case of DIV, both image and audio modalities—enabling a more thorough and realistic evaluation of our approach. Performance is assessed by projecting the data into lower-dimensional spaces using the proposed methods and applying a standard classifier, with results compared against a selection of state-of-the-art techniques. To provide a comprehensive performance comparison, we include a baseline model in which the raw data is directly used as input to the classifier. Across all methods, the recognition rate (IR) is adopted as the primary evaluation metric. The recognition rate is computed using a 1-nearest neighbor (1-NN) classifier, based on the Euclidean distance between the projected training and test samples. To ensure a fair and consistent comparison, we utilize the supervised variants of the evaluated methods. Gaussian weights are employed, with the Gaussian kernel parameter set to half the median of the dataset, as recommended in \cite{kokiopoulou2009enhanced}. All experiments were conducted using MATLAB 2021a on a standard computing platform equipped with a 2.1 GHz Intel Core i7 (8th Gen) processor and 8 GB of RAM.

\subsection{Compared methods}
We refer to the proposed methods as $MDA_e^{rt}$, for Ratio-Trace method, and $MDA_e^{tr}$ for the Trace-Ratio based on the iterative method. The subscript $_\textbf{e}$ stands for the Einstein product. We compared our methods with the following state-of-the-art methods:
\begin{itemize}
\item The LDA method.
\item Orthogonal Locality Preserving Projections $OLPP_e$, Orthogonal Neighborhood Preserving Projections $ONPP_e$ via the Einstein product \cite{zahir2024higher} that are, respectively, the generalized higher order methods of the matrix versions of (OLPP), and (ONPP).
\item $MDA_t^{tr}$ \cite{dufrenois2023multilinear} which is a multiview generalization of LDA using the T-product for third order tensors.
\end{itemize}
We have concentrated our experiments on linear methods to mitigate the out-of-sample problem and ensure a fair comparison. It’s important to note that we have utilized the supervised versions of the methods, which includes all approaches except for the unsupervised method PCA. A supervised version relies on labeled data for graph construction rather than learning it directly from the data. In the supervised methods, Gaussian weights are applied using the class labels, with the Gaussian parameter set to half the median of the dataset, as recommended in \cite{kokiopoulou2009enhanced}. 
The maximum number of iteration is set to $100$. The parameter of the regularized methods is set to $0.01$. The convergence criteria in the iterative methods is set to $1e-9$.\\
In accordance with the paper’s guidelines, methods requiring matrix inputs reshape each image of size $H \times W$ pixels (height $H$, width $W$) with three RGB channels into a vector of length $H \cdot W \cdot 3$, forming a data matrix of size $(H \times W \times 3, n)$, where $n$ is the number of samples. Conversely, for $MDA_t^{tr}$ the data is reshaped into $(n, H \times W, 3)$ to match the $t$-product format.

\subsection{Dataset}
We will use multiple datasets with varying sizes, sample counts, class distributions, and types to provide a broader perspective on the proposed method.
\begin{itemize}
\item \textbf{MNIST dataset} \footnote{\url{https://lucidar.me/en/matlab/load-mnist-database-of-handwritten-digits-in-matlab/}} consists of 60,000 training images and 10,000 testing images featuring labeled handwritten digits. Each image is formatted as a $28 \times 28$ pixel gray scale image, and all images are normalized to ensure uniform intensity level. We randomly selected 1000 images for training and 200 for testing. The data input in this case is a tensor of size $28 \times 28 \times 1000$ for the training set and $28 \times 28 \times 200$ for the testing set.

\item \textbf{GTDB dataset}\footnote{\url{https://www.anefian.com/research/face_reco.htm}} includes 750 RGB images of 50 distinct individuals. Each individual is represented by 15 images, capturing a variety of facial expressions, scales, and lighting conditions. The dataset is resized to $60 \times 60$ pixels. We randomly select 12 images per individual for training and the remaining 3 for testing. The data input in this case is a tensor of size $60 \times 60 \times 3 \times 500$ for the training set and $60 \times 60 \times 3 \times 250$ for the testing set.

\item \textbf{DIV dataset} was introduced in \cite{dufrenois2023multilinear}. This dataset covers digits from 0 to 9 using two modalities—visual and audio—drawn respectively from the MNIST and FSDD \footnote{\url{https://github.com/Jakobovski/free-spoken-digit-dataset}} datasets. MNIST provides 60,000 training and 10,000 test grayscale images of handwritten digits, each at a resolution of $28\times 28$ pixels. FSDD contains 500 audio recordings (8 kHz) of spoken English digits, each lasting on average about 0.5 seconds. The dataset is harmonized into a tensor of size $64 \times 64 \times 2 \times 5000$ where 5000 is the number of samples chosen randomly from the original two dataset. We select randomly 4000 samples for training and 1000 for testing.

\item \textbf{WDCM dataset} is a hyperspectral image of Washington DC Mall, captured by the Hyperspectral Digital Imagery Collection Experiment (HYDICE) \cite{he2015hyperspectral}. This dataset consists of 191 spectral bands and has a spatial resolution of $1208 \times 307$ pixels. Based on this image, we define four distinct land cover classes: grassland, tree, roof, and road. These classes are manually delineated using $7\times 7$ pixel blocks. In total, 8032 blocks are extracted, distributed as follows: 1894 blocks for tree, 1919 for grassland, 2616 for roof, and 1603 for road. To facilitate analysis, a tensor representation is constructed, organizing the data into a four-dimensional array of size $7\times 7 \times191 \times 8032$. This tensor structure enables more effective multilinear analysis and feature extraction. The modifications and enhancements to the dataset and its representation were inspired by the work of \cite{dufrenois2023multilinear}, which introduced refinements to the prepossessing and classification methodology.

\end{itemize}
A visual summary of the datasets is shown in the following Figure.
\begin{figure}[h]
\centering
\begin{minipage}[t]{0.5\textwidth}
\includegraphics[width=0.9\textwidth]{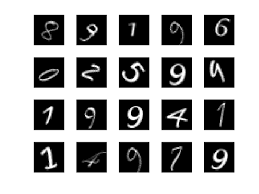}
\end{minipage}
\begin{minipage}[t]{0.3\textwidth}
\includegraphics[width=0.85\textwidth]{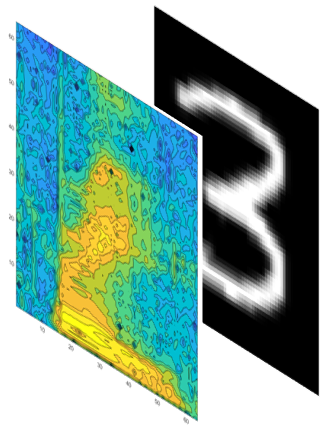}
\end{minipage}
\begin{minipage}[t]{0.65\textwidth}
\includegraphics[width=0.9\textwidth]{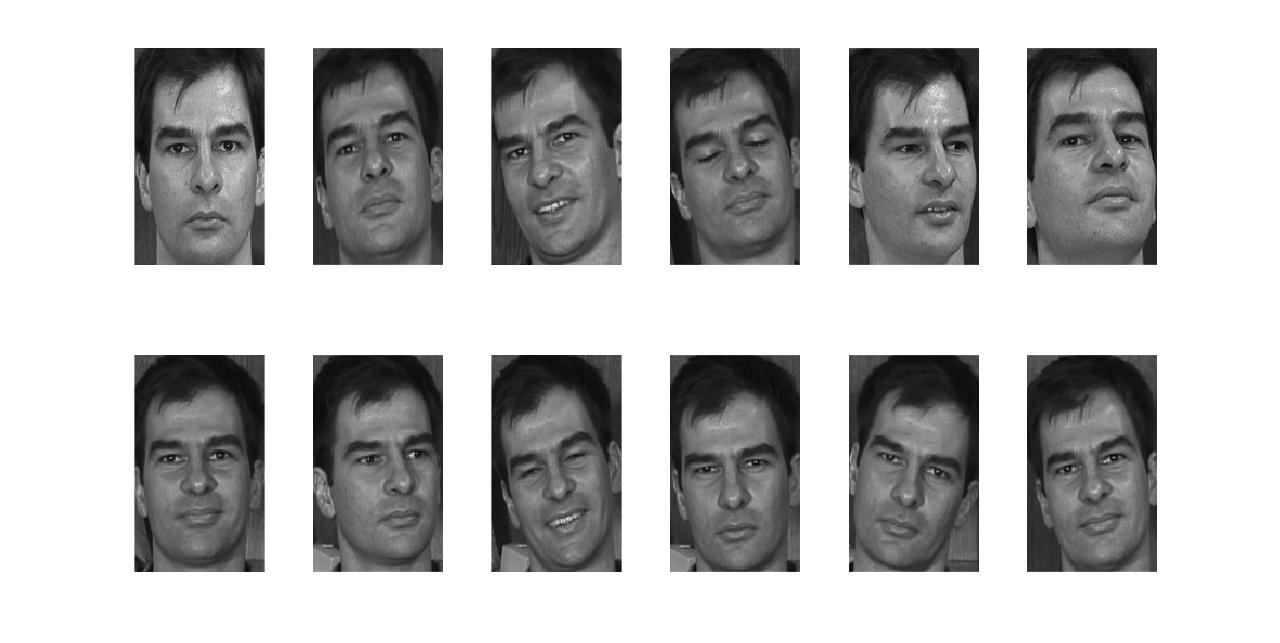}
\end{minipage}
\begin{minipage}[t]{0.29\textwidth}
\includegraphics[width=0.9\textwidth]{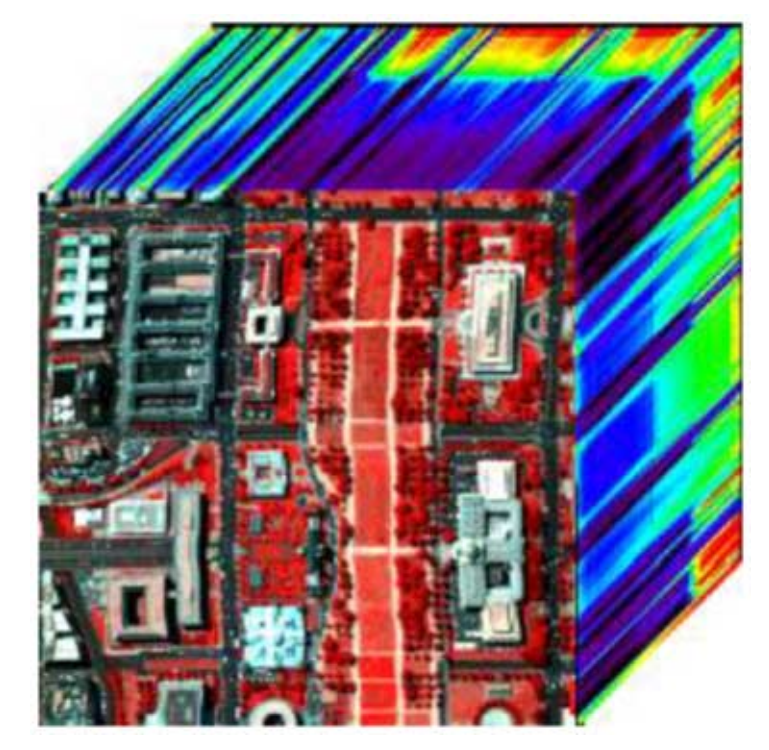}
\end{minipage}
\caption{Top left: MNIST dataset; Top right: DIV dataset; Bottom left: Random person from the GTDB dataset; Bottom right: WDCM dataset.}
\label{fig:GTDB}
\end{figure}
In the GTDB dataset, we will concentrate on comparing the methods that are based on the Tensorial format, along with LDA.
\subsection{Results}
This section presents the results of the approaches on the multiple datasets, across different subspace dimensions, as well as the associated time complexity.
The performance of the proposed method is evaluated using the recognition rate as the evaluation metric. It is defined as the ratio of correctly classified test samples to the total number of test samples. A test sample is considered correctly classified if it has the minimum distance to its corresponding training sample in the projected feature space. All recognition rates are computed on the test set.\\
Table \ref{Tab:combined_mnist_div_wdcm} report on the results of MNIST,DIV and DCM datasets, respectively. In both cases, the proposed method consistently outperforms existing approaches, confirming its effectiveness and robustness across different types of data.\\
In all tables, the best recognition rate for each subspace dimension is highlighted in \textbf{bold}, while the second best is \underline{underlined}.

\begin{table}[h]
\centering
\resizebox{\textwidth}{!}{
\begin{tabular}{c|c|cccccc}
\textbf{Dataset} & $d$ 
& \textbf{LDA} 
& $\boldsymbol{OLPP_e}$ 
& $\boldsymbol{ONPP_e}$ 
& $\boldsymbol{MDA_t^{tr}}$ 
& $\boldsymbol{MDA_e^{tr}}$ 
& $\boldsymbol{MDA_e^{rt}}$ \\
\hline
\multirow{8}{*}{MNIST}
& 5 & \textbf{75.50} & 50.50 & 56.00 & \underline{69.50} & 47.00 & \underline{69.50} \\
& 10 & 77.00 & 75.50 & \textbf{81.50} & 75.50 & 76.00 & 75.50 \\
& 15 & 77.50 & 81.00 & 80.50 & 76.50 & \textbf{82.00} & 76.50 \\
& 20 & 78.50 & 85.00 & 83.50 & 78.00 & \textbf{84.50} & 78.00 \\
& 25 & 78.00 & 86.00 & \textbf{87.50} & 75.50 & \textbf{87.50} & 75.50 \\
& 30 & 76.50 & 85.50 & \textbf{87.50} & 76.00 & \underline{86.50} & 76.00 \\
& 35 & 76.50 & 88.00 & \textbf{89.50} & 75.50 & \underline{89.00} & 75.50 \\
& 40 & 76.50 & 88.00 & \textbf{89.00} & 78.00 & \underline{88.50} & 78.00 \\
\hline
\multirow{5}{*}{DIV}
& 2 & \textbf{75.50} & 50.50 & 56.00 & \underline{69.50} & 47.00 & \underline{69.50} \\
& 4 & 77.00 & 75.50 & \textbf{81.50} & 75.50 & 76.00 & 75.50 \\
& 6 & 77.50 & 81.00 & 80.50 & 76.50 & \textbf{82.00} & 76.50 \\
& 8 & 78.50 & 85.00 & 83.50 & 78.00 & \textbf{84.50} & 78.00 \\
& 10 & 78.00 & 86.00 & \textbf{87.50} & 75.50 & \textbf{87.50} & 75.50 \\
\hline
\multirow{3}{*}{WDCM}
& 2 & \textbf{75.50} & 50.50 & 56.00 & \underline{69.50} & 47.00 & \underline{69.50} \\
& 3 & 77.00 & 75.50 & \textbf{81.50} & 75.50 & 76.00 & 75.50 \\
& 4 & 77.50 & 81.00 & 80.50 & 76.50 & \textbf{82.00} & 76.50 \\
\hline
\end{tabular}
}
\caption{Recognition rates (\%) of various methods across different target dimensions $d$ on the MNIST, DIV, and WDCM datasets (columns reordered). The baseline performance for each dataset is as follows: MNIST — 8.5\%, DIV — 26.60\%, and WDCM — 51.51\%.}
\label{Tab:combined_mnist_div_wdcm}
\end{table}

\begin{table}[h]
\centering
\resizebox{\textwidth}{!}{
\begin{tabular}{c|cccccccc}
\backslashbox{$d$}{\small{Method}} 
& $\boldsymbol{LDA}$ 
& $\boldsymbol{LDA_{r}}$ 
& $\boldsymbol{OLPP_e}$ 
& $\boldsymbol{ONPP_e}$ 
& $\boldsymbol{MDA_{t}^{tr}}$ 
& $\boldsymbol{MDA_{e}^{tr}}$ 
& $\boldsymbol{MDA_{e}^{rt}}$ 
& $\boldsymbol{MDA_{er}^{tr}}$ \\
\hline
5 & \underline{64.00} & \underline{64.00} & 38.67 & 44.00 & 42.67 & 75.33 & \textbf{76.60} & 75.33 \\
10 & 72.67 & \underline{78.00} & 56.00 & 61.33 & 48.67 & \textbf{86.00} & \textbf{86.00} & \textbf{86.00} \\
15 & 74.00 & \underline{82.00} & 60.00 & 67.33 & 60.00 & \underline{86.00} & \textbf{86.67} & \underline{86.00} \\
20 & 72.00 & \underline{84.00} & 63.33 & 69.33 & 60.00 & \textbf{88.00} & \textbf{88.00} & \textbf{88.00} \\
25 & 70.00 & \underline{84.67} & 68.67 & 72.00 & 68.67 & \textbf{88.00} & \textbf{88.00} & \textbf{88.00} \\
30 & 66.00 & \underline{84.67} & 73.33 & 74.00 & 61.33 & 85.33 & \textbf{86.00} & \textbf{86.00} \\
35 & 67.33 & \underline{84.67} & 77.33 & 75.33 & 66.67 & 87.33 & \textbf{88.00} & \textbf{88.00} \\
40 & 66.00 & \underline{86.00} & 80.00 & 76.67 & 66.67 & \underline{88.00} & \textbf{88.67} & \underline{88.00} \\
\hline
\end{tabular}}
\caption{Recognition rates (\%) on different target dimensions $d$ on the GTDB dataset. Baseline is 75.00\%.}
\label{Tab:GTDB}
\end{table}

\noindent
Table \ref{Tab:GTDB} presents the classification performance of various methods on the GTDB dataset. The results clearly demonstrate the superiority of the proposed method, which achieves a notably high recognition rate even at relatively low dimensions (starting from 
$d=10$). These findings also highlight the general advantage of higher-order methods over their lower-dimensional counterparts, particularly in capturing complex structures in the data. 
An additional observation concerns the regularization effect in LDA-based methods. While regularization significantly improves LDA performance, its impact appears to be less pronounced in higher-dimensional settings and comparatively less necessary for MDA, which already performs well without it. Moreover, as shown in the first figure, the methods exhibit stronger performance on RGB images, suggesting that they are particularly effective when applied to high-dimensional data with richer feature representations.
\medskip
\begin{table}[h]
\centering
\begin{tabular}{|c|cccc|}
\backslashbox{$d$}{Method} & $\boldsymbol{MDA_e^{rt}}$ & $\boldsymbol{MDA_e^{tr}}$ & $\boldsymbol{LDA_r^{tr}}$ & $\boldsymbol{MDA_e^{rt,ls}}$\\
\hline
5 & 57.05 & 482.66 & 16.06 & 23.07\\
10 & 50.46 & 423.79 & 16.28 & 22.79\\
\end{tabular}
\caption{Time of different proposed methods on the GTDB dataset in seconds.}
\label{Tab:GTDB_time}
\end{table}
\noindent
For generalized eigenvalue problems, complexity is $O(n^3)$ in the matrix case. In the tensor setting, operations scale with mode sizes; iterative TR requires repeated eigenvalue computations, roughly $10\times$ slower in practice. Using partial solvers (e.g., Lanczos or \texttt{eigs}) mitigates this. Thus RT is preferable for large-scale problems when runtime is critical, while TR is preferable for accuracy.\\
Table \ref{Tab:GTDB_time} reports the computational time of the various methods on the GTDB dataset. It can be observed that the dimensionality of the subspace does not significantly affect the runtime. This behavior stems from the implementation strategy, which employs the \textit{eig} function to compute the full spectrum of eigenvalues, followed by selection of the top components. In contrast, using the \textit{eigs} function—designed to compute only the largest $d$ eigenvalues—would substantially reduce execution time, especially in high-dimensional scenarios.\\
These results are consistent with the theoretical complexity of the methods. In particular, iterative algorithms incur a higher computational cost due to the need for eigenvalue decomposition at each iteration. For this dataset, convergence is typically reached within approximately 10 iterations, which explains the observed increase in computation time—roughly 10 times longer than non-iterative approaches.\\
Moreover, the findings highlight a common limitation of high-dimensional methods: the increased computational burden of eigenvalue computations. Although TR and RT methods achieve similar performance levels, they differ in runtime, which could influence the choice of method in real-world applications.\\
While the proposed methods introduce additional computational overhead, they exhibit strong capabilities in capturing low-dimensional feature representations. This presents a favorable trade-off between accuracy and runtime in contexts where precision is more critical than speed.

Finally, this work could be extended in the spirit of \cite{zhang2009generalized}, by incorporating the exponential of a tensor with respect to the Einstein product. This exponential is straightforward to define and offers a valuable advantage: it guarantees to overcome the undersampling problem. Such an extension could further enhance the robustness and applicability of the proposed framework.

\section{Conclusion}\label{sec:conclusion}
In this work, we introduced Multilinear Discriminant Analysis via the Einstein product ($MDA_e$), a novel approach for dimensionality reduction and classification. The method is formulated as a Trace Ratio optimization problem, aiming to compute a projection tensor that effectively maximizes the ratio between: the between-class and the within-class scatter. We established the existence and uniqueness of the solution (up to a unitary transformation) and proposed an iterative algorithm for its computation, incorporating a regularization term to guarantee stability and solvability. 
Extensive experiments were conducted on the MNIST, GTDB, and DIV datasets, benchmarking $MDA_e$ against state-of-the-art methods. The results demonstrate that $MDA_e$ consistently outperforms existing techniques, particularly in scenarios where computational time is not a limiting factor. Future research will explore optimization of computational efficiency and the extension of the method to broader application domains.
\\

\noindent {\bf Declaration of competing interest}
The authors declare that they have no competing financial interests or personal relationships that could have appeared to influence the work reported in this paper.

\bibliographystyle{siam}
\bibliography{cas-refs}

\end{document}